\documentclass[10pt, reqno]{amsart}
\usepackage{amsrefs, amsmath, amssymb, amsthm, graphicx,marvosym}
\usepackage{cancel}

\def\veps{\varepsilon}
\def\vp{\varphi}

\def\eq#1{(\ref{#1})}
\def\nn{\nonumber}
\def\({\left(\begin{array}{cccccc}}
\def\){\end{array}\right)}

\def\eq#1{(\ref{#1})}
\def\nn{\nonumber}

\def\({\left(\begin{array}{cccccc}}
\def\){\end{array}\right)}

\def\bes{\begin{eqnarray}}
\def\ees{\end{eqnarray}}

\newcommand{\ba}{\overset{\raisebox{0pt}[0pt][0pt]{\text{\raisebox{-.5ex}{\scriptsize$\leftharpoonup$}}}}}
\newcommand{\fa}{\overset{\raisebox{0pt}[0pt][0pt]{\text{\raisebox{-.5ex}{\scriptsize$\rightharpoonup$}}}}}

\newcommand{\del}{\partial}

\newcommand{\beq}{\begin{equation}}
\newcommand{\eeq}{\end{equation}}
\newcommand{\bea}{\begin{eqnarray}}
\newcommand{\eea}{\end{eqnarray}}
\newcommand{\beann}{\begin{eqnarray*}}
\newcommand{\eeann}{\end{eqnarray*}}

\newcommand{\lam}{\ensuremath{\lambda}}

\newcommand{\RR}{\mathbb{R}}

\newcommand{\bp}{\begin{proof}}
\newcommand{\ep}{\end{proof}}
\newcommand{\nquad}{\negthickspace\negthickspace
\negthickspace\negthickspace}
\newcommand{\nqquad}{\nquad\nquad}

\DeclareMathOperator{\sgn}{sgn}
\DeclareMathOperator{\var}{Var}

\newtheorem{theorem}{Theorem}[section]
\newtheorem{proposition}[theorem]{Proposition}

\newtheorem{lemma}[theorem]{Lemma}
\newtheorem{definition}[theorem]{Definition}

\newtheorem{remark}[theorem]{Remark}

\numberwithin{equation}{section}

\begin{document}

\title{TVD fields and isentropic gas flow}

\author{Geng Chen}
\address{G.~Chen, Department of Mathematics, Penn State University, University Park, 
State College, PA 16802, USA ({\tt chen@math.psu.edu}).}

\author{Helge Kristian Jenssen}
\address{ H.~K.~Jenssen, Department of Mathematics, Penn State University, University Park, 
State College, PA 16802, USA ({\tt jenssen@math.psu.edu}).}
\thanks{H.~K.~Jenssen was partially supported by NSF grants DMS-1009002.}

\date{\today}
\begin{abstract}
	Little is known about global existence of large-variation solutions to Cauchy problems 
	for systems of conservation laws in one space dimension. Besides results for
	$L^\infty$ data via compensated compactness, the existence of global BV solutions 
	for arbitrary BV data remains an outstanding open problem. 
	In particular, it is not known if isentropic gas dynamics admits an a priori variation bound 
	which applies to all BV data.
		
	In a few cases such results are available: scalar equations,
	Temple class systems, $2\times 2$-systems satisfying Bakhvalov's condition,
	and, in particular, isothermal gas dynamics. In each of these cases 
	the equations admit a TVD (Total Variation Diminishing) field: 
	a scalar function defined on state space whose spatial 
	variation along entropic solutions does not increase in time. 
	
	In this paper we consider strictly hyperbolic $2\times 2$-systems and derive a representation  
	result for scalar fields that are TVD across all pairwise wave interactions, when the latter are  
	resolved as in the Glimm scheme. We then use this to show that isentropic gas 
	dynamics with a $\gamma$-law pressure function does not admit any nontrivial 
	TVD field of this type.
\end{abstract}

\maketitle
\tableofcontents



\section{Introduction}
Consider the one-dimensional Cauchy problem for the $p$-system in Lagrangian 
coordinates $(t,X)$:
\bea
	\tau_t-u_X &=& 0\label{tau}\\
	u_t+p(\tau)_X &=&0.\label{mom}
\eea
We are primarily interested in the cases 
\beq\label{p_syst}
	p(\tau)=\tau^{-\gamma},\qquad \gamma=const.\geq 1.
\eeq
In gas dynamics $\tau$ denotes specific volume, $u$ is the particle velocity, $p$ is pressure,
and $\gamma$ is the adiabatic exponent. The initial data $\tau_0$ and $u_0$ are prescribed,
and the state space is the open, right half-plane $\RR^+\times\RR=\{(\tau,u)\,|\, \tau>0\}$.

The case $\gamma=1$ models isothermal gas flow, while $\gamma>1$ models isentropic
flow of an ideal, polytropic gas. The former case was resolved by Nishida \cite{nish_68} who 
established global-in-time 
existence of weak, entropic solutions for any BV-data. The proof employs the Glimm scheme
\cite{gl} and provides a uniform variation bound for the solution by exploiting translation invariance 
of the shock curves when $\gamma=1$. Bakhvalov \cite{bakh_70} generalized Nishida's analysis 
by formulating conditions on wave interactions in systems of the form \eq{tau}-\eq{mom} which 
guarantee uniform BV-bounds. These conditions do not hold for \eq{tau}-\eq{p_syst}.

Several large variation results have been established for isentropic flow, see \cites{ns,diperna_73}. 
These require that $(\gamma-1)\times \var(\tau_0,u_0)$ is sufficiently small. Results of the
latter type have been established for the full Euler system as well, see 
\cites{liu_77a,liu_77b,temple_81}. Also, Glimm \& Lax
\cite{glimm_lax_70} established global existence and decay for a class of systems (including 
isentropic gas dynamics) requiring only that the {\em oscillation} of the initial data be small.
Finally, the method of compensated compactness has been used to establish 
global weak solutions for arbitrary $L^\infty$ data, \cites{diperna_83a, diperna_83b, dcl_89}.

On the other hand it remains an open problem to establish, or rule out, uniform variation 
bounds for solutions \eq{tau}-\eq{p_syst}, when $\gamma>1$, 
for general BV data. By this we mean bounds of the form\footnote{Throughout the paper 
``$\var$'' denotes total variation with respect to the spatial variable $X$.}  
\[\var U(t,\cdot)\leq C_0\qquad\text{for all times $t\geq 0$,}\]
where $C_0$ depends only on the initial data and $U=(\tau,u)$.

For near-equilibrium solutions, Glimm's theorem \cite{gl} establishes such an estimate for 
systems of conservation laws, provided the initial data have sufficiently small total variation. 
The proof relies on the so-called Glimm functional, a non-local quantity which is made 
to decrease in time through a careful balance of increase in variation against decrease in 
potential for future interaction. 

In contrast to this let us consider three cases where existence results for large-variation data 
are available without the use of a Glimm functional:
\begin{enumerate}
	\item Scalar equations: $v_t+f(v)_X=0$, $v(t,x)\in \RR$;
	\item Temple class systems: i.e.\ systems equipped with a full set of Riemann 
	coordinates and such that shock curves and rarefaction curves coincide
	(Temple \cite{temple_83} introduced these systems and characterized all possible 
	$2\times 2$-systems of this type; electrophoresis \cite{daf} provides an example); 
	\item Isothermal gas dynamics ($\gamma=1$), and more generally
	systems satisfying Bakhvalov's conditions (see \cites{nish_68,bakh_70,diperna_73,liu_77b}).
\end{enumerate}
For a scalar equation it is well-known that the total variation of the solution is non-increasing
in time \cite{daf}. Temple class systems enjoy the same property provided the variation of the solution
is measured correctly, namely in terms of changes in the Riemann coordinates. For 
isothermal gas-dynamics, which is not of Temple class, Nishida \cite{nish_68} defined the functional that 
records variation of one Riemann coordinate $r$ across backward shocks, plus variation of another
Riemann coordinate $s$ across forward shocks, i.e.\ 
\beq\label{nish_fncl}
	N(t):=\underset{\ba S}{\var}\, r(t) + \underset{\fa S}{\var}\, s(t).
\eeq 
It turns out that, for the specific system of isothermal gas dynamics, and more generally for systems
satisfying Bakhvalov's conditions \cite{bakh_70}, the property that $N(t)$ is non-increasing is equivalent 
to the statement that the functional 
\beq\label{liu_fncl}
	L(t):=\var(s(t)-r(t))
\eeq 
is non-increasing; see Theorem 2.3 in \cite{liu_77b}.

By exploiting the existence of scalar fields whose variation is non-increasing along solutions, 
global existence for general BV data has been established for each of the cases (1)-(3). 
(In the case of (3) it is also  important that the Riemann invariants vary in opposite directions 
across rarefactions and shocks.) It is therefore natural to ask if such TVD\footnote{I.e.\ ``Total Variation 
Diminishing;'' this allows for fields whose variation is just non-increasing in time.} fields exist also for isentropic flow. 
The main results of the present work are, roughly stated, the following.
\begin{itemize}
	\item[I.] Representation result: any TVD field $\vp$ for a strictly hyperbolic 
	$2\times 2$-system is of the form $\theta(s)-\psi(r)$, where $r$ and $s$ are 
	Riemann invariants; see Theorem \ref{repr}.
	\item[II.] Non-existence result: the system \eq{tau}-\eq{p_syst} for 
	isentropic gas dynamics, with $\gamma>1$, does not admit any non-trivial TVD field; 
	see Theorem \ref{non_exist}.
\end{itemize}
In what follows we provide a precise definition of TVD fields for general systems. 
We then establish part I above under some mild regularity conditions on the TVD field. 
Finally we consider isentropic gas dynamics and prove part II.
Both results are based on a study of pairwise wave interactions and apply to ``Glimm-type'' TVD 
fields; see Section \ref{tvd_fields}.

\medskip

\section{TVD fields for strictly hyperbolic $2\times 2$ - systems}
\subsection{Systems and assumptions}\label{syst}
We consider a strictly hyperbolic system of two equations for the 
unknown $U(t,X)\in\mathcal U^\text{open}\subset\RR^{2\times 1}$:
\beq\label{claw}
	U_t+f(U)_X=0, \qquad t\geq 0,\, X\in\RR.
\eeq
Let the diagonalization of $Df$ be 
\beq\label{e_struct}
	Df(U)=\left[\, \ba R(U)\,\vline\, \fa R(U)\,\right]
	\left[\begin{array}{cc}
	\ba \lam(U) & 0\\
	0 & \fa \lam(U)
	\end{array}\right]
	\left[\begin{array}{c}
	\ba L(U)\\
	\hline
	\fa L(U)
	\end{array}\right],
\eeq
where $\ba \lam(U)<\fa \lam(U)$ denote the slow and fast characteristic speeds. The corresponding
right eigenvectors $\ba R,\, \fa R\in\RR^{2\times 1}$ and left eigenvectors 
$\ba L,\, \fa L\in\RR^{1\times 2}$ satisfy
\[\ba R\cdot \fa L\equiv 0, \qquad \fa R\cdot \ba L\equiv 0.\]
For the following terminology we refer to \cite{daf}.
We assume that each characteristic field of \eq{claw} is either genuinely 
non-linear or linearly degenerate throughout $\mathcal U$. We will also assume that any 
Riemann problem $(U_l,U_r)$, with $U_l,U_r\in\mathcal U$, has a unique, self-similar 
solution consisting of one slow elementary wave (i.e., shock, contact, or centered rarefaction)
and one fast elementary wave connecting $U_l$ to $U_r$ through a middle state. 
The eigenvectors are normalized according to 
\beq\label{nrmlzd}
	\nabla \ba\lam\cdot\ba R\geq 0,\qquad \nabla \fa\lam \cdot\fa R\geq 0,\qquad 
	\ba R \cdot\ba L\equiv 1,\qquad \ \fa R \cdot\fa L\equiv 1.
\eeq
We denote the corresponding wave curves through a base point (left state) 
$\bar U$ by $\ba W(\veps;\bar U)$ and $\fa W(\veps;\bar U)$, $\veps \in\RR$. The parameter
value $\veps$ is also referred to as the {\em strength} of the waves $(\bar U,\ba W(\veps;\bar U))$
and $(\bar U,\fa W(\veps;\bar U))$. In this section 
we assume that the wave curves are parametrized such that: 
\begin{itemize}
	\item $\ba W(0;\bar U)=\bar U$,
	\item $\veps>0$ gives the states $\ba W(\veps;\bar U)$ on the right of backward 
	rarefaction waves with left state $\bar U$,
	\item $\veps<0$ gives the states $\ba W(\veps;\bar U)$ on the right of backward shock
	waves with left state $\bar U$.
\end{itemize}
Similarly for $\fa W(\veps;\bar U)$. For now we do not specify further the parametrizations: 
$\ba R$ and $\fa R$ denote any fixed versions of the eigen-fields (subject to \eq{nrmlzd}). 
We shall later choose these to commute.

For any fixed choice of the eigen-fields $\ba R$, $\fa R$ we introduce Riemann coordinates $r$ 
and $s$ satisfying 
\beq\label{riem1}
	\nabla r\cdot\ba R\equiv 1,\qquad \nabla r \cdot \fa R\equiv 0,
\eeq
\beq\label{riem2}
	\nabla s \cdot\ba R\equiv 0,\qquad \nabla s \cdot\fa R\equiv 1.
\eeq
(See \eq{riem_coords_p_syst} for the case of the $p$-system \eq{tau}-\eq{p_syst}.) 
Finally, the map $(r,s)\mapsto U(r,s)$ is assumed to be a diffeomorphism.

\subsection{Exact and Glimm-type TVD fields}\label{tvd_fields}
For clarification we introduce two types of TVD fields, exact and Glimm-type. 
However, both the representation result (Theorem \ref{repr}) and the 
non-existence result for the $p$-system (Theorem \ref{non_exist}) will be formulated
for Glimm-type TVD fields only.

First, $U(t,X)$ is an {\em entropic weak solution} to \eq{claw} provided it is a weak (distributional)
solution of \eq{claw}, and satisfies 
\[\eta(U)_t+q(U)_X\leq 0 \qquad\text{in $\mathcal D'$,}\]
whenever $\eta,\, q:\mathcal U\to \RR$ are such that $\nabla_U q=\nabla_U\eta D_Uf$ 
and $\eta$ is convex. (We refer to \cites{daf} for details.) Given a scalar field 
$\vp:\mathcal U\to \RR$ and a weak solution $U(t,X)$ of \eq{claw} such that 
$U(t)\equiv U(t,\cdot)\in\text{BV}(\RR)$, the variation of the 
function $X\mapsto\vp(U(t,X))$ is denoted $\var \vp(U(t))$, or just $\var\vp$.
\begin{definition}\label{exact_tvd}
	A scalar field $\vp:\mathcal U\to \RR$ is an {\em exact TVD field} for the system \eq{claw}
	provided the following holds: whenever $U(t,X)$ is an entropic weak solution of \eq{claw}
	with $U(t)\in\emph{BV}(\RR)$ for $t>0$, 
	the map $t\mapsto\var \vp(U(t))$ is non-increasing on $\RR_+$.
\end{definition}
Next, we recall how pairwise wave interactions are resolved in the Glimm scheme \cite{gl}. 
Consider two approaching elementary waves (shocks, contacts, or centered rarefactions) 
with left, middle, and right states $U_l$, $U_m$, and $U_r$, respectively. The interaction 
of these waves is resolved by solving the {\em interaction Riemann problem} $(U_l,U_r)$,
whose solution we denote by $U(t,x)$ with $t>t^*$, $t^*$ being the time of interaction.
\begin{definition}\label{glimm_tvd}
	A scalar field $\vp:\mathcal U\to \RR$ is a {\em Glimm-type TVD field} for the system \eq{claw}
	provided the following holds: whenever $U_l$, $U_m$, and $U_r$ are left, middle, and right 
	states, respectively, in a pairwise interaction as described above, the spatial variation of  
	$\vp(U(t,X))$ for $t>t^*$ is less than or equal 
	to its spatial variation across the incoming waves $(U_l,U_m)$ and $(U_m,U_r)$ combined.
\end{definition}
\noindent We next derive a necessary condition on the functional form of any Glimm-type TVD field 
for \eq{claw}.

\subsection{Representation of Glimm-type TVD fields}
We shall obtain the general form of a Glimm-type TVD field in terms of Riemann coordinates 
by considering weak, head-on interactions. To analyze these we use the Glimm interaction 
estimate (\cite{daf}, Section 9.9) together with the freedom of choosing commuting versions 
of the eigenvector fields $\ba R$ and $\fa R$.

Changing notation slightly we let the extreme left state in the interactions be
$\bar U$. In the following computation all quantities are evaluated at $\bar U$
unless indicated otherwise,. 
The far right state is denoted $U$, while the incoming middle state and the outgoing 
middle state are denoted $U_0$ and $\hat U$, respectively. 
Consider the head-on interaction on the left in Figure 1. 
\begin{figure}\label{interactns2}
	\centering
	\includegraphics[width=9cm,height=2.5cm]{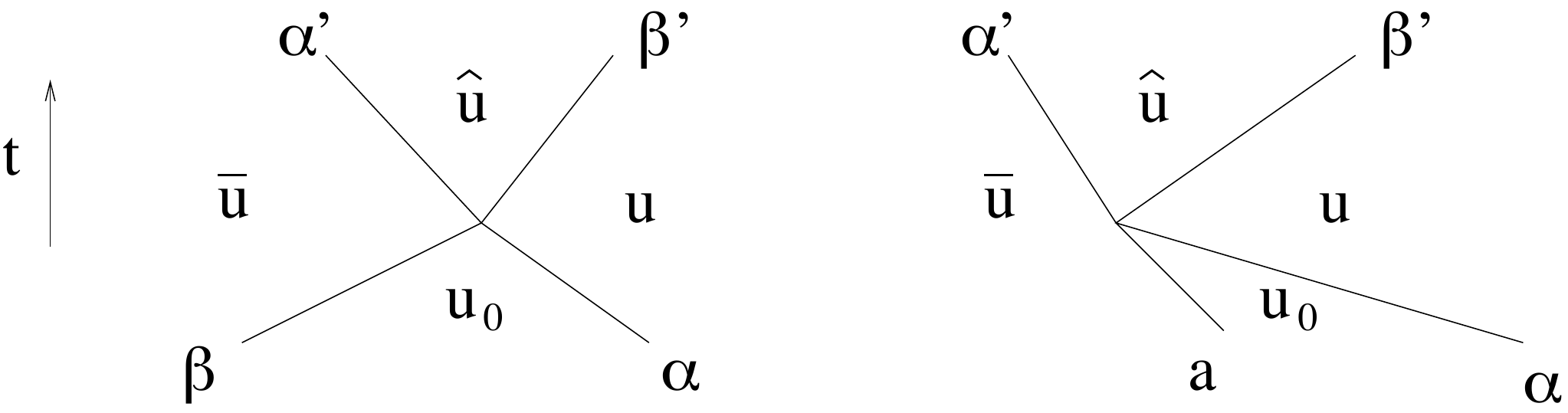}
	\caption{Pairwise interactions: head-on (I) and overtaking (II). (Schematic)}
\end{figure} 
Let $\alpha$, $\alpha'$, $\beta$, and $\beta'$ denote the parameter values corresponding to incoming 
and outgoing waves in the slow and fast families, respectively. The total variation of $\vp(U)$ 
before and after interaction are, respectively,
\beq\label{var_phi_before}
	\var \vp^-\equiv |\vp(U)-\vp(U_0)|+|\vp(U_0)-\vp(\bar U)|,
\eeq
and
\beq\label{var_phi_after}
	\var \vp^+\equiv |\vp(U)-\vp(\hat U)|+|\vp(\hat U)-\vp(\bar U)|.
\eeq
To evaluate these for weak interactions we use Taylor expansion and
Glimm's interaction estimate \cites{gl, daf}.
Consider the variation of $\vp$ before interaction. The middle state $U_0$ is given by
\beq\label{U_0}
	U_0=\fa W(\beta;\bar U)=\bar U +\beta \fa R+\frac{1}{2}\beta^2(D\fa R)\fa R+O^3(\alpha),
\eeq
while the right state $U$ is given by
\begin{align}
	U&=\ba W(\alpha;U_0)=U_0 +\alpha \ba R\big|_{U_0}
	+\frac{1}{2}\alpha^2(D\ba R)\ba R\big|_{U_0}+O(\alpha^3)\nn\\
	&=U_0+\alpha\ba R+\alpha\beta(D\ba R)\fa R
	+\frac{1}{2}\alpha^2(D\ba R)\ba R+O^3(\alpha,\beta),
	\label{U}
\end{align}
where we have used \eq{U_0}. We use  $O^3(\alpha,\beta)$ to indicate third order terms such as
$\alpha^3$, $\alpha^2\beta$ etc., and recall that all quantities are evaluated at 
$\bar U$ unless otherwise indicated.

Taylor expanding $\vp$ about $U_0$, and using \eq{U_0} and \eq{U}, we obtain
\begin{align}
	& \vp(U)-\vp(U_0)\nn\\
	&=\alpha\nabla\vp\ba R+\left[\frac{\alpha^2}{2}\nabla\vp(D\ba R)\ba R
	+\frac{\alpha^2}{2}\ba R^T\!\!(D^2\vp)\ba R+\alpha\beta\nabla \vp(D\ba R)\fa R 
	+\alpha\beta \fa R^T\!\!(D^2\vp)\ba R\right] + O^3(\alpha,\beta).\label{before11}
\end{align}
(Note: gradients are row vectors, while $\ba R$ and $\fa R$ are column vectors. A `$T$' superscript
denotes transpose, and $D^2\vp$ denotes the Hessian of $\vp$. Here and below $\nabla$ and $D^2$ 
are with respect to $U$ unless stated otherwise.) Similarly we have
\beq\label{before12}
	\vp(U_0)-\vp(\bar U)=\beta\nabla\vp\fa R+\frac{\beta^2}{2}\left[\nabla\vp(D\fa R)\fa R
	+\fa R^T\!\!(D^2\vp)\fa R\right] + O^3(\alpha,\beta).
\eeq
Thus
\begin{align}
	\var \vp^- =&\quad \left|\alpha\nabla\vp\ba R+\frac{\alpha^2}{2}\left[\nabla\vp(D\ba R)\ba R
	+\ba R^T\!\!(D^2\vp)\ba R\right]+\alpha\beta\left[\nabla \vp(D\ba R)\fa R 
	+\fa R^T\!\!(D^2\vp)\ba R\right] + O^3(\alpha,\beta)\right|\nn\\
	& +\left|\beta\nabla\vp\fa R+\frac{\beta^2}{2}\left[\nabla\vp(D\fa R)\fa R
	+\fa R^T\!\!(D^2\vp)\fa R\right] + O^3(\alpha,\beta)\right| \label{TV_before}
\end{align}
Next, for the states after interaction we first use similar Taylor expansions with respect to the 
outgoing strengths $\alpha'$ and $\beta'$, to obtain a corresponding expression for $\var \vp^+$. 
We shall then recall the Glimm interaction estimate which provides 
$\alpha'$ and $\beta'$ in terms of $\alpha$ and $\beta$ to leading orders. 
Combining these yields an expression for $\var  \vp^+$ in terms of $\alpha$ and $\beta$ 
which we can compare to $\var  \vp^-$ as given by \eq{TV_before}.
As above we have 
\[\hat U=\ba W(\alpha';\bar U)=\bar U +\alpha' \ba R+\frac{1}{2}{\alpha'}^2(D\ba R)\ba R+O^3(\alpha'),\]
and
\[ U=\hat U +\beta' \fa R+\alpha'\beta'(D\fa R)\ba R+\frac{1}{2}{\beta'}^2(D\fa R)\fa R+O^3(\alpha',\beta').\]
It follows that
\begin{align*}
	& \vp(U)-\vp(\hat U)\nn\\
	& =\beta'\nabla\vp\fa R+\alpha'\beta'\left[\nabla\vp(D\fa R)\ba R + \ba R^T\!\!(D^2\vp)\fa R\right]
	+\frac{{\beta'}^2}{2}\left[\nabla\vp (D\fa R)\fa R+\fa R^T\!\!(D^2\vp)\fa R\right]+ O^3(\alpha',\beta'),
\end{align*}
and 
\[\vp(\hat U)-\vp(\bar U) = \alpha'\nabla\vp\ba R+\frac{{\alpha'}^2}{2}
\left[\nabla\vp (D\ba R)\ba R+\ba R^T\!\!(D^2\vp)\ba R\right]+ O^3(\alpha',\beta').\]
Hence,
\begin{align}
	\var \vp^+ =&\quad \left|\beta'\nabla\vp\fa R+\alpha'\beta'\left[\nabla\vp(D\fa R)\ba R + \ba R^T\!\!(D^2\vp)\fa R\right]
	+\frac{{\beta'}^2}{2}\left[\nabla\vp (D\fa R)\fa R+\fa R^T\!\!(D^2\vp)\fa R\right]+ O^3(\alpha',\beta')\right|\nn\\
	&+\left|\alpha'\nabla\vp\ba R+\frac{{\alpha'}^2}{2}
	\left[\nabla\vp (D\ba R)\ba R+\ba R^T\!\!(D^2\vp)\ba R\right]+ O^3(\alpha',\beta')\right|. \label{TV_after}
\end{align}
We now invoke Glimm's interaction estimate (\cite{daf}, Section 9.9) according to which
\[\alpha'=\alpha+\alpha\beta\ba L\big[\fa R,\ba R\big]+O^3(\alpha,\beta)\qquad\text{and}\qquad
\beta'=\beta+\alpha\beta\fa L\big[\fa R,\ba R\big]+O^3(\alpha,\beta),\]
where $[X,Y]\equiv (DY)X-(DX)Y$ denotes the commutator. We then scale $\fa R,\ba R$ such that their
commutator vanishes. (This is always possible for planar vector fields.) 
With these versions of $\ba R$ and $\fa R$ fixed we obtain that 
\[\alpha'=\alpha+O^3(\alpha,\beta)\qquad\text{and}\qquad\beta'=\beta+O^3(\alpha,\beta),\]
and substitution into \eq{TV_after} yields
\begin{align}
	\Delta \var  \vp &\equiv \var  \vp^+ -\var  \vp^- \nn\\
	=&\quad \left|\beta\nabla\vp\fa R+\alpha\beta\left[\nabla\vp(D\fa R)\ba R + \ba R^T\!\!(D^2\vp)\fa R\right]
	+\frac{{\beta}^2}{2}\left[\nabla\vp (D\fa R)\fa R+\fa R^T\!\!(D^2\vp)\fa R\right]+ O^3(\alpha,\beta)\right|\nn\\
	&+\left|\alpha\nabla\vp\ba R+\frac{{\alpha}^2}{2}
	\left[\nabla\vp (D\ba R)\ba R+\ba R^T\!\!(D^2\vp)\ba R\right]+ O^3(\alpha,\beta)\right|\nn\\
	& -\left|\alpha\nabla\vp\ba R+\frac{\alpha^2}{2}\left[\nabla\vp(D\ba R)\ba R
	+\ba R^T\!\!(D^2\vp)\ba R\right]+\alpha\beta\left[\nabla \vp(D\ba R)\fa R 
	+\fa R^T\!\!(D^2\vp)\ba R\right] + O^3(\alpha,\beta)\right|\nn\\
	& -\left|\beta\nabla\vp\fa R+\frac{\beta^2}{2}\left[\nabla\vp(D\fa R)\fa R
	+\fa R^T\!\!(D^2\vp)\fa R\right] + O^3(\alpha,\beta)\right|.\label{dlt_tv}
\end{align}
We now use that $\big[\fa R,\ba R\big]\equiv 0$ and $\ba R^T\!\!(D^2\vp)\fa R\equiv \fa R^T\!\!(D^2\vp)\ba R$, 
and set
\begin{align}
	A&:=\nabla\vp\fa R\nn\\
	B&:=\nabla\vp(D\fa R)\ba R + \fa R^T\!\!(D^2\vp)\ba R\nn\\
	C&:= \nabla\vp (D\fa R)\fa R+\fa R^T\!\!(D^2\vp)\fa R\nn\\
	D&:=\nabla\vp\ba R\nn\\
	E&:=\nabla\vp (D\ba R)\ba R+\ba R^T\!\!(D^2\vp)\ba R,\nn
\end{align}
such that \eq{dlt_tv} reads
\begin{align}
	\Delta \var  \vp =&\quad \left|\beta A+\alpha\beta B+\frac{{\beta}^2}{2}C+ O^3(\alpha,\beta)\right| 
	+\left|\alpha D+\frac{{\alpha}^2}{2}E+ O^3(\alpha,\beta)\right|\nn\\
	&-\left|\alpha D+\frac{\alpha^2}{2}E+\alpha\beta B + O^3(\alpha,\beta)\right|
	-\left|\beta A+\frac{\beta^2}{2}C + O^3(\alpha,\beta)\right|\nn\\
	&=: |\Delta_1|+|\Delta_2|-|\Delta_3|-|\Delta_4|.\label{dlt_tv2}
\end{align}

Recall that $A$, $B$, $C$, $D$, $E$ are all evaluated at $\bar U$. We consider 
the situation when both $A=\nabla\vp \cdot\fa R$ and $D=\nabla\vp \cdot\ba R$ are non-vanishing 
on some open set $\mathcal U'\subset\mathcal U$.  
With $\bar U$ fixed in $\mathcal U'$ we then consider weak interactions. Specifically, we 
choose $|\alpha|,\, |\beta|\neq 0$ and so small that 
\[\sgn \Delta_1=\sgn (\beta A)=\sgn \Delta_4,\]
and
\[\sgn \Delta_2=\sgn (\alpha D)=\sgn \Delta_3.\]
As we assume $A$ and $D$ are non-vanishing there are four possibilities:
\begin{itemize}
	\item[(i)] $\beta A$, $\alpha D>0$
	\item[(ii)] $\beta A$, $\alpha D<0$
	\item[(iii)] $\beta A<0<\alpha D$
	\item[(iv)] $\alpha D<0<\beta A$.
\end{itemize}
For these we get, respectively,
\begin{itemize}
	\item[(i)'] $\Delta \var  \vp=\Delta_1+\Delta_2-\Delta_3-\Delta_4
	=O^3(\alpha,\beta)$
	\item[(ii)'] $\Delta \var \vp=-\Delta_1-\Delta_2+\Delta_3+\Delta_4
	=O^3(\alpha,\beta)$
	\item[(iii)'] $\Delta \var \vp=-\Delta_1+\Delta_2-\Delta_3+\Delta_4
	=-2\alpha\beta B+O^3(\alpha,\beta)$
	\item[(iv)'] $\Delta \var \vp=\Delta_1-\Delta_2+\Delta_3-\Delta_4
	=+2\alpha\beta B+O^3(\alpha,\beta)$.
\end{itemize}
Now, assume for contradiction that $B\neq 0$. 
Whatever the signs of $A$ and $D$ are, we consider a head-on interaction with 
left state $\bar U$, and with $\alpha$-value $\alpha_1\neq 0$ and $\beta$-value
$\beta_1\neq 0$ chosen so that Case (iii) holds. By reducing further (if necessary)
$|\alpha_1|$ and $|\beta_1|$, we obtain from (iii)' that
\beq\label{contr1}
	-2\alpha_1\beta_1 B<0.
\eeq
We then consider the head-on interaction with the same left state $\bar U$, but now 
with $\alpha$- and $\beta$-values $\alpha_2:=-\alpha_1$ and $\beta_2:=-\beta_1$.
This latter interaction then belongs to Case (iv), and (iv)' gives
\[2\alpha_2\beta_2 B=2\alpha_1\beta_1 B<0.\]
This contradicts \eq{contr1} and shows that we must have $B=0$.
Thus, in order that the total $X$-variation of $\vp(U(t,X))$ be non-increasing 
across all head-on Glimm interactions, it is necessary that
\[B=\nabla\vp(D\fa R)\ba R + \fa R^T\!\!(D^2\vp)\ba R=0,\]
which is equivalent to
\beq\label{nec_cond_2}
	\nabla(\nabla\vp\cdot\fa R)\cdot\ba R=0.
\eeq
\begin{remark}\label{no_overatk}
	Performing the same type of analysis for weak, overtaking interactions shows
	that $\Delta \var \vp=O^3(\alpha,\beta)$ in all such interactions.
	Thus, no similar constraint follows from non-increase of variation across weak,
	overtaking interactions.
\end{remark}
For the versions of $\ba R$ and $\fa R$ fixed above we let $r$ and $s$ be Riemann 
coordinates satisfying \eq{riem1} and \eq{riem2}.
By changing to these coordinates and considering $\vp$ as a function of $r$ and $s$,
the necessary condition \eq{nec_cond_2} takes the simple form
\beq\label{nec_cond_3}
	\del_{rs}^2\vp=0.
\eeq
This shows that $\vp$ must admit a representation of the form
\beq\label{r_s_repr}
	\vp(r,s)=\theta(s)-\psi(r)
\eeq
in any convex subset of $\mathcal V:=\{\,(r,s)\,|\, U(r,s)\in\mathcal U'\,\}$. 
(The minus sign is chosen for convenience in formulating later results.)
This conclusion was reached under the assumption that $\nabla_U\vp\cdot \fa R$ and $\nabla_U\vp \cdot\ba R$
are both non-vanishing in $\mathcal U'$. We note that if one of these vanishes identically 
in $\mathcal U'$, then $\vp$ is again of the form \eq{r_s_repr} in $\mathcal U'$ (with $\theta$ or $\psi$ vanishing).
We summarize our findings:
\begin{theorem}\label{repr}
	Given a $2\times 2$-system \eq{claw} satisfying the standard structural assumptions in Section \ref{syst}. 
	Let its eigen-structure be given by \eq{e_struct}, with $\ba R$ and $\fa R$ scaled to commute
	$(\big[\fa R,\ba R\big]\equiv 0)$ in $\mathcal U$, and let $r$, $s$ denote the corresponding Riemann 
	coordinates satisfying \eq{riem1}-\eq{riem2}.
	 
	Assume that $\vp\in C^2(\mathcal U)$ is a Glimm-type TVD field for \eq{claw}, and 
	consider $\vp$ as a function of $r$ and $s$. Define the sets
	\[\mathcal V:=\{\,(r,s)\in \RR^2\,|\, U(r,s)\in\mathcal U\,\}\]
	and
	\[\mathcal V':=\{\,(r,s)\in\mathcal V\,|\, \del_r\vp(r,s)\neq 0\text{ and } \del_s\vp(r,s)\neq 0\,\}.\]
	Then $\vp$ admits a representation of the form $\vp(r,s)=\theta(s)-\psi(r)$ on any open, convex subset of $\mathcal V'$. 
	Furthermore, if $\mathcal V$ is convex and $\mathcal V'$ is dense in $\mathcal V$, then there exist 
	functions $\theta(s)$, $\psi(r)$ such that $\vp$ has this representation throughout $\mathcal V$.
\end{theorem}
\begin{proof}
	It only remains to argue for the last part. So assume $\mathcal V'$ is dense in $\mathcal V$ and 
	that the latter is convex. Given any point $(\bar r,\bar s)$ in the open set $\mathcal V'$, we choose a ball
	$B\subset\mathcal V'$ about $(\bar r,\bar s)$. According to the analysis above $\del^2_{rs}\vp$ vanishes identically on $B$,
	and in particular $\del^2_{rs}\vp(\bar r,\bar s)=0$. As $\del^2_{rs}\vp$ is continuous on $\mathcal V$ 
	we conclude that $\del^2_{rs}\vp\equiv 0$ throughout $\mathcal V$. Finally, by convexity of $\mathcal V$
	it follows that there is a single pair of functions $\theta$ and $\psi$ such that $\vp(r,s)=\theta(s)-\psi(r)$
	throughout $\mathcal V$.\footnote{Some restriction on the domain is required for the existence of {\em one} 
	pair $\theta,\, \psi$ such that \eq{r_s_repr} holds in all of $\Omega$. E.g., on $\Omega:=(-2,2)^2\setminus[-1,1]^2$ let 
	$\vp(r,s)=\eta(s)\chi_{(0,1)\times (-1,1)}(r,s)$, where $\eta\not\equiv0$ is smooth 
	and with support in $(-1,1)$. Then $\del_{rs}^2\vp\equiv 0$ on $\Omega$, but $\vp$ does not 
	have one representation of the form \eq{r_s_repr} which is valid throughout $\Omega$.}
\end{proof}
\begin{remark}
	Note that the decay of the Liu functional $L(t)$ in \eq{liu_fncl} shows that $\vp(r,s)=s-r$ is  
	a Glimm-type TVD field for any $2\times 2$-systems satisfying Bakhvalov's conditions.
\end{remark}

\section{Isentropic gas dynamics}
In the remainder of this paper we consider the specific case of isentropic gas dynamics
with a standard $\gamma$-law pressure function \eq{tau}-\eq{p_syst}.

\subsection{Eigen-structure and wave curves}
The characteristic speeds are 
\beq\label{char_speeds}
	\ba\lambda=-\sqrt{\gamma}\,\,\tau^{-\alpha-1}\qquad \text{and}\qquad 
	\fa\lambda=\sqrt{\gamma}\,\,\tau^{-\alpha-1},
\eeq
where 
\[\alpha:=\frac{\gamma-1}{2}.\]
As corresponding right eigenvectors of the Jacobian of the flux $(-u,p(\tau))^T$ we choose
\[\ba R={\textstyle\frac{1}{2}}\left[
\begin{array}{c}
	\frac{1}{\sqrt{\gamma}}\tau^{\alpha+1}\\
	1
\end{array}\right]
\qquad \text{and}\qquad 
\fa R={\textstyle\frac{1}{2}}\left[
\begin{array}{c}
	-\frac{1}{\sqrt{\gamma}}\tau^{\alpha+1}\\
	1
\end{array}\right].\]
These satisfy $\nabla_{\!(\tau,u)}\ba \lambda\cdot \ba R>0$, 
$\nabla_{\!(\tau,u)}\fa \lambda\cdot \fa R>0$, and $[\fa R,\ba R]\equiv 0$.
As Riemann invariants we use
\beq\label{riem_coords_p_syst}
	r:=u-\kappa\tau^{-\alpha}\qquad \text{and}\qquad s:=u+\kappa\tau^{-\alpha},
	\qquad \kappa:=\frac{\sqrt{\gamma}}{\alpha},
\eeq
such that
\[\nabla_{\!(\tau,u)}r \cdot \ba R\equiv 1,\qquad 
\nabla_{\!(\tau,u)}r \cdot \fa R\equiv 0,\qquad \nabla_{\!(\tau,u)}s \cdot \ba R\equiv 0,\qquad
\nabla_{\!(\tau,u)}s \cdot \fa R\equiv 1.\]
We parametrize the wave curves using $\xi$-ratios where
\beq\label{xi_ratio}
	\xi:=\rho^\alpha=\tau^{-\alpha};
\eeq
see Remark \ref{det}. The Riemann invariants are then
\beq\label{s_r_def}
	r=u-\kappa  \xi,\quad s=u+\kappa  \xi.
\eeq
We next give the parametrizations in the $(\xi,u)$-plane of the backward and forward 
wave curves of the ``first type.'' That is, given a base point $(\bar \xi,\bar u)$, we consider
the curves of points $(\xi,u)$ such that the Riemann problem with left state $(\bar \xi,\bar u)$
and right state $(\xi,u)$ yields a single backward or forward wave (rarefaction or entropic 
shock). Letting $b$ and $f$ denote the $\xi$-ratios $\xi_\text{right}/\xi_\text{left}$ across 
backward and forward waves, respectively, the parametrizations are given by:
\begin{align}
	&\text{Backward waves:}\qquad
	{\ba{V}}(b;\bar \xi,\bar u) = \left(\!\!\begin{array}{c}
		b\bar \xi \\
		\bar u -{\ba{\phi}}(b)\bar \xi\\
	\end{array} \!\!\right)
	\qquad\left\{\begin{array}{ll}
	\ba R:\quad 0<b<1 \\
	\ba S:\quad b>1
	\end{array}\right.\label{bkwd_wave}\\
	&\text{Forward waves:}\,\,\,\qquad
	{\fa{V}}(f;\bar \xi,\bar u) = \left(\!\!\begin{array}{c}
		f\bar \xi \\
		\bar u +{\fa{\phi}}(f)\bar \xi\\
	\end{array} \!\!\right)
	\qquad\left\{\begin{array}{ll}
	\fa R:\quad f>1 \\
	\fa S:\quad 0<f<1,
	\end{array}\right.\label{frwd_wave}
\end{align}
where the auxiliary functions $\ba\phi$ and $\fa\phi$ are given by
\[{\ba{\phi}}(x) = \left\{\begin{array}{ll}
	\kappa  (x-1) \quad & 0<x<1\\\\
	\sqrt{(1-x^{-\frac{1}{\alpha}})(x^{\frac{\gamma}{\alpha}}-1)}\quad & x>1
\end{array}\right\}\] 
\[{\fa{\phi}}(x) = \left\{\begin{array}{ll}
	-\sqrt{(1-x^{-\frac{1}{\alpha}})(x^{\frac{\gamma}{\alpha}}-1)} \quad & 0<x<1\\\\
	\kappa  (x-1) \quad  & x>1.
\end{array}\right\}\] 
We refer to the $\xi$-ratios $b$ and $f$ as wave strengths.
Note that the auxiliary functions $\ba\phi$ and $\fa\phi$ are both strictly increasing and satisfy
\beq\label{phi_reln}
	\fa \phi(x)=-x{\ba{\phi}}\big(\textstyle\frac{1}{x}\big)\,.
\eeq
A calculation shows that the map
$x\mapsto \sqrt{(1-x^{-\frac{1}{\alpha}})(x^{\frac{\gamma}{\alpha}}-1)}$
is convex up for $x>1$. It follows that $\ba\phi$ is convex up, while
$\fa\phi$ is convex down. Furthermore, it is immediate to verify that 
\beq\label{phi_prime_infty}
	\lim_{x\to+\infty}\ba\phi '(x)=+\infty\qquad\text{and}\qquad 
	\lim_{x\to0^+}\fa\phi(x)=-\infty.
\eeq
For later reference we also record the fact that 
\beq\label{phi_limit}
	x-\left(\frac{\ba\phi(x)+\kappa(x+1)}{\ba\phi'(x)+\kappa}\right)\to+\infty
	\qquad\text{as $x\to+\infty$.}
\eeq
\begin{figure}\label{ba_fa_phi}
	\centering
	\includegraphics[width=8cm,height=7cm]{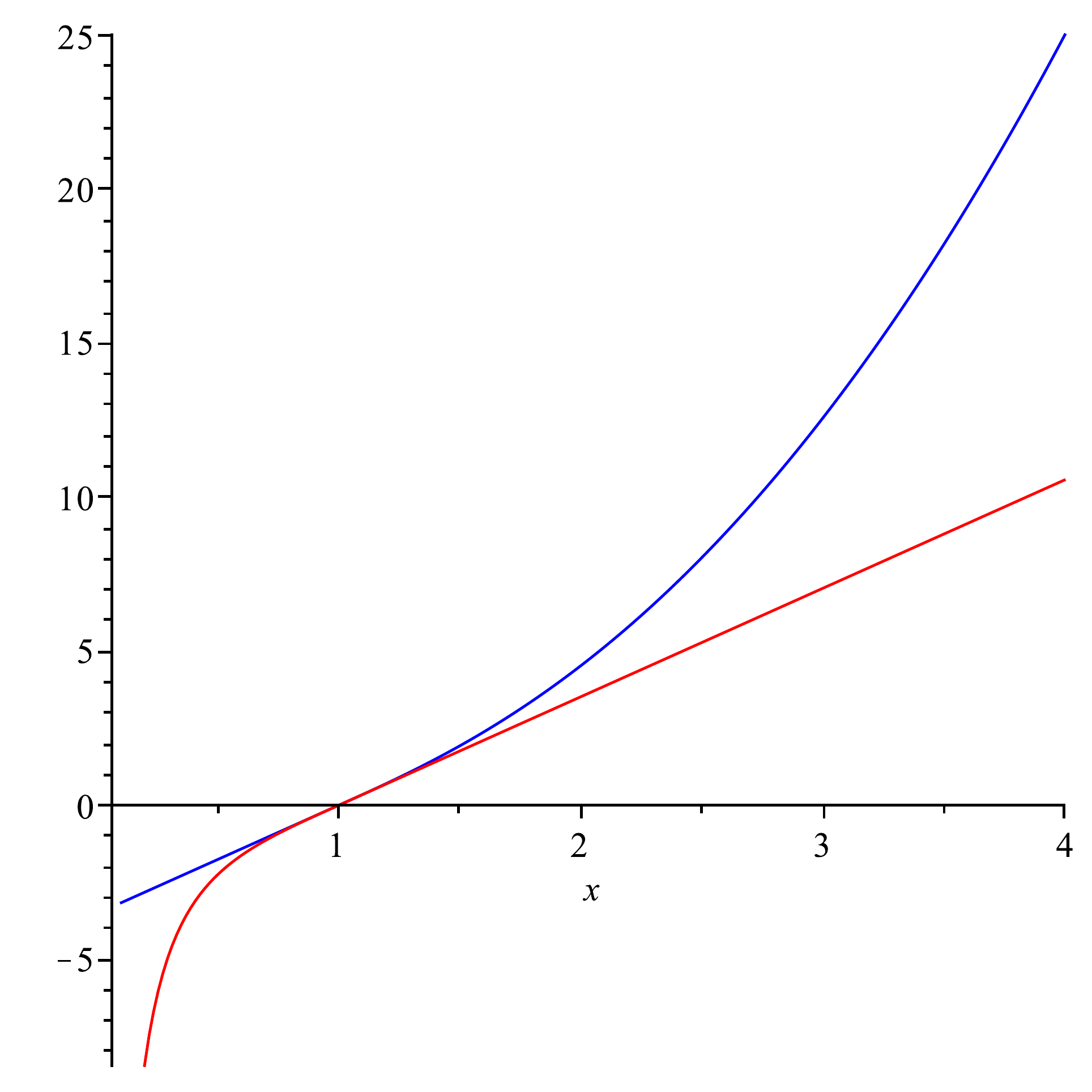}
	\caption{The auxiliary $\phi$ -functions (backward=solid, forward=dashed)}
\end{figure} 
The wave curves in the $(r,s)$-plane are illustrated in Figure 3. Note that the no-vacuum
region $\{\,\rho>0\,\}$ corresponds to the half-plane $\mathcal H:=\{\,s>r\,\}$ in $(r,s)$-coordinates.
%
\begin{figure}\label{riem_coords}
	\centering
	\includegraphics[width=8cm,height=6cm]{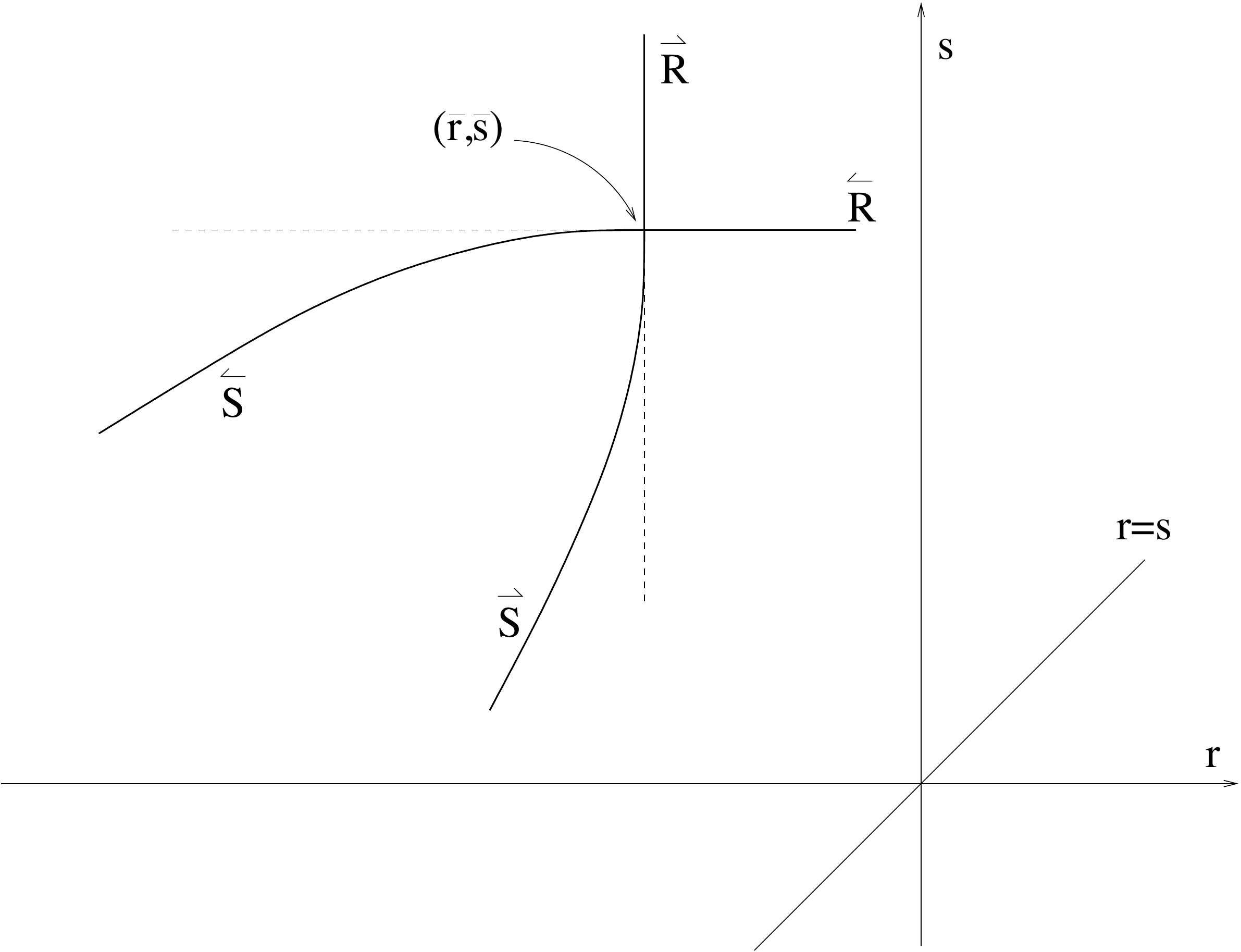}
	\caption{The wave curves in the Riemann coordinate plane}
\end{figure} 
%
%

\subsection{Riemann problems and vacuum criterion}
Consider the Riemann problem with left state $(\bar\xi,\bar u)$ and right state $(\xi,u)$,
and let $b$ and $f$ denote the $\xi$-ratios of the resulting backward and forward waves,
respectively. A short calculation shows that $b$ is given as the root of
\beq\label{b_eqn}
	\bar\xi\ba\phi(b)+\xi\ba\phi\left(\textstyle\frac{b\bar\xi}{\xi}\right)=\bar u- u,
\eeq
and $f=\frac{\xi}{b\bar\xi}$. As the left-hand side of \eq{b_eqn} is increasing with respect to $b$, it follows that
the Riemann problem $\big((\bar\xi,\bar u), \, (\xi,u)\big)$ has a unique solution without vacuum provided
\beq\label{no_vac}
	u-\bar u<\kappa  (\xi+\bar\xi)\qquad\qquad\text{(no vacuum)}.
\eeq

\subsection{Change in scalar functions of the Riemann invariants across shocks}
For later reference we record how scalar fields $h(r)$ and $k(s)$ change across a shock wave.
Let $(\bar u,\bar \xi)$ denote the state on the left of a backward shock wave with strength (i.e., 
$\xi$-ratio) $x$. The state on the right of the shock is then $(\tilde \xi,\tilde u)$, where
\[\tilde \xi =x\bar\xi,\qquad \tilde u=\bar u-\ba\phi(x)\bar \xi,\]
and the corresponding values of the Riemann invariants are
\[\bar r=\bar u-\kappa\bar\xi,\qquad \bar s=\bar u+\kappa\bar\xi
\qquad\text{(left state)},\]
and 
\[\tilde r=\bar u-\ba\phi(x)\bar \xi-\kappa x\bar\xi,\qquad
\tilde s=\bar u-\ba\phi(x)\bar \xi+\kappa x\bar\xi\qquad\text{(right state)}.\]
We then have
\beq\label{delta_h}
	\Delta_x h(r):= h(\tilde r) - h(\bar r)
	=-\bar\xi\int_1^x h'(\bar u-\ba\phi(\sigma)\bar \xi-\kappa \sigma\bar\xi)
	(\ba\phi'(\sigma)+\kappa)\, d\sigma,
\eeq
and
\beq\label{delta_k}
	\Delta_x k(s):= k(\tilde s) - k(\bar s)
	=-\bar\xi\int_1^x k'(\bar u-\ba\phi(\sigma)\bar \xi+\kappa \sigma\bar\xi)
	(\ba\phi'(\sigma)-\kappa)\, d\sigma.
\eeq
Similar identities hold for changes across forward waves.

\section{Pairwise interactions in isentropic flow}\label{interactions}
There are six essentially distinct types of pairwise wave interactions:
\beq\label{gr1}
	\left.
	\begin{array}{ll}
		\mbox{Ia}: & {\fa{R}}{\ba{R}}\\
		\mbox{Ib}: & {\fa{R}}{\ba{S}}\\
		\mbox{Ic}: & {\fa{S}}{\ba{S}}
	\end{array}	
	\qquad\right\}\qquad \mbox{head-on interactions}
\eeq
\beq\label{gr2}
	\left.
	\,\,\,\begin{array}{ll}
		\mbox{IIa}: & {\ba{S}}{\ba{S}}\\
		\mbox{IIb}: & {\ba{S}}{\ba{R}}\\
		\mbox{IIc}: & {\ba{R}}{\ba{S}}
	\end{array}	
	\qquad\right\}\qquad \mbox{overtaking interactions.}
\eeq
The head-on interaction 
\[\mbox{Ib'}: \fa{S}\ba{R},\]
and the overtaking interactions 
\beq\label{gr2'}
	\begin{array}{ll}
		\mbox{IIa'}: & {\fa{S}}{\fa{S}}\\
		\mbox{IIb'}: & {\fa{R}}{\fa{S}}\\
		\mbox{IIc'}: & {\fa{S}}{\fa{R}},
	\end{array}	
\eeq
are qualitatively the same as those in Ib and IIa-IIc, respectively. 
Finally, for cases IIb, IIc, IIb', and IIc', there are two possible outcomes depending 
on the relative strengths of the incoming waves.

The analysis of interaction Riemann problems for the $p$-system has been 
carried out, see \cite{chang_hsiao}.
In cases Ic, IIa, and IIa', where both incoming waves are shocks, this gives the exact,
weak entropy solution of the wave interaction. When one of the incoming waves is a rarefaction, 
the actual solution involves penetration of a rarefaction wave, to which the interaction Riemann 
problem provides an approximate solution. For completeness we include a brief description
of the results.

\begin{figure}\label{interactns1}
	\centering
	\includegraphics[width=8cm,height=2.7cm]{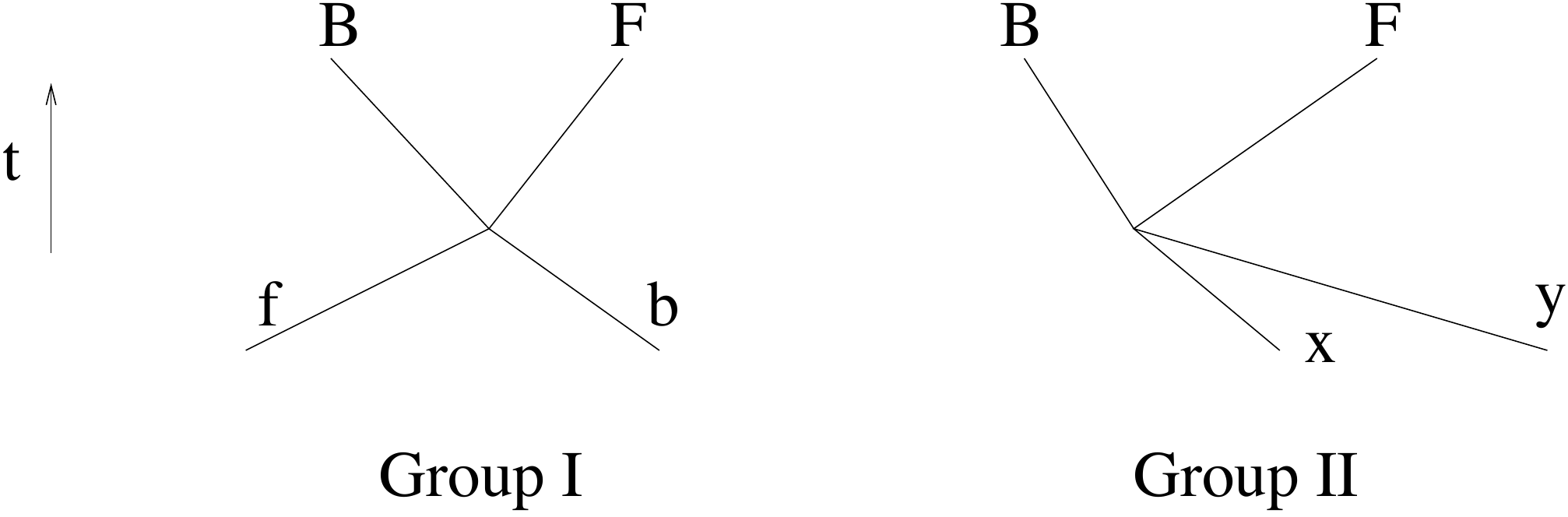}
	\caption{Interactions: head-on (I) and overtaking (II). (Schematic)}
\end{figure} 
%
%

\subsection{Group I: Head-on interactions} 
Let the state on the far left be $(\bar\xi,\bar u)$ and let the $\xi$-ratios of the 
incoming (outgoing) backward and forward waves be $b$ and $f$ ($B$ and $F$), 
respectively (see Figure 4, left diagram). Traversing the waves before and after interaction shows 
that $B$ and $F$ satisfy
\begin{align}
	BF &= bf\\
	\ba\phi(B)-B\fa\phi(F) &=f\ba\phi(b)-\fa\phi(f).
\end{align}
\begin{remark}\label{det}
	Parametrizing the wave curves in terms of $\xi$-ratios as in \eq{bkwd_wave} and 
	\eq{frwd_wave} implies that $B$ and $F$ are determined independently of 
	the left state $(\bar\xi,\bar u)$ in the interaction. The same is true for overtaking interactions. 
	This is advantageous when we later want to search for 
	interactions where new variation is created; see Section \ref{2nd_part}.
\end{remark}
\noindent Applying \eq{phi_reln} we get that the $\xi$-ratio $B=B(b,f)$ solves the equation
\beq\label{head_on_B}
	\mathcal{F}(B;b,f):=\ba\phi(B)+bf\ba\phi\left(\frac{B}{bf}\right)+\fa\phi(f)-f\ba\phi(b)=0.
\eeq
The condition for having no vacuum is that $B>0$, or
\[\ba\phi(b)+\kappa  b>\frac{1}{f}\left(\fa\phi(f)-\kappa  \right)
\qquad\text{(no vacuum in head-on)}.\]
A direct calculation shows that there is no vacuum in Ib (${\fa{R}}{\ba{S}}$) and Ic 
(${\fa{S}}{\ba{S}}$) interactions, while there is no vacuum in Ia (${\fa{R}}{\ba{R}}$) 
interaction if and only if the incoming ratios satisfy $b+\frac{1}{f}>1$.

By analyzing the interactions in the $(r,s)$-plane of Riemann invariants, we obtain that:
Ia interactions always yield $\ba R\fa R$; Ib interactions always yield $\ba S\fa R$ (and
Ib' interactions always yield $\ba R\fa S$); Ic interactions always yield $\ba S\fa S$.

\subsection{Group II: Overtaking interactions}\label{overtaking_interactions}
Let the state on the far left be $(\bar\xi,\bar u)$ and let the $\xi$-ratios of the 
incoming  backward waves be $x$ (leftmost) and $y$ (rightmost).
As above let $B$ and $F$ denote the $\xi$-ratios of the outgoing backward and forward 
waves, respectively. Traversing the waves before and after interaction yields
\begin{align}
	BF &= xy\label{overtaking1}\\
	\ba\phi(B)-B\fa\phi(F) &=\ba\phi(x)+x\ba\phi(y).\label{overtaking2}
\end{align}
Using \eq{phi_reln} we get that the $\xi$-ratio $B$ solves the equation
\beq\label{overtaking_B}
	\mathcal{G}(B;x,y):=\ba\phi(B)+xy\ba\phi\left(\frac{B}{xy}\right)-\ba\phi(x)-x\ba\phi(y)=0.
\eeq
The condition for having no vacuum is that $B>0$, or
\[{\textstyle\kappa  }(1+xy)+\ba\phi(x)+x\ba\phi(y)>0
\qquad\text{(no vacuum in overtaking)}.\]
A direct calculation shows that there is never vacuum formation in any of the overtaking interactions.
Furthermore, by analyzing the interactions in the $(r,s)$-plane of Riemann invariants, we obtain that:
IIa ($\ba S\ba S$) interactions always yield $\ba S\fa R$; 
IIb ($\ba S\ba R$) interactions yield either $\ba S\fa S$ (when the incoming shock is strong relative 
to the incoming rarefaction), or $\ba R\fa S$ (when the incoming shock is weak relative to the 
incoming rarefaction); IIc ($\ba R\ba S$) interactions yield either $\ba S\fa S$ (when the incoming 
shock is strong relative to the incoming rarefaction), or $\ba R\fa S$ (when the incoming shock is 
weak relative to the incoming rarefaction). Similar statements apply to IIa', IIb', and IIc' interactions.

\medskip

\section{Non-existence of Glimm-type TVD fields for isentropic flow}
The representation of Glimm-type TVD fields in Theorem \ref{repr} 
was obtained by analyzing only head-on interactions. To show that no such field exists 
for the $p$-system we will consider specific overtaking interactions.
We work in the Riemann coordinates $r$, $s$, given in \eq{riem_coords_p_syst},
which vary over
\[\mathcal H=\{\,(r,s)\,|\, s>r \,\}.\]
We restrict attention to non-degenerate fields:
\begin{definition}
	A scalar field $\vp:\mathcal H\to \RR$ is called {\em non-degenerate} provided
	the set
	\beq\label{non_deg} 
		\mathcal H_\vp:=\{\,(r,s)\in\mathcal H\,\,|\,\, \del_r\vp(r,s)\neq 0\text{ or } 
		\del_s\vp(r,s)\neq 0\,\} \quad\text{is dense in $\mathcal H$.}
	\eeq
\end{definition}
\noindent Our main result is the following:
\begin{theorem}\label{non_exist}
	There is no non-degenerate $C^2$-smooth TVD field of Glimm type for 
	\eq{tau}-\eq{p_syst}, with $p>1$, which is defined on all of $\mathcal H$.
\end{theorem}
\noindent 
Towards the proof of Theorem \ref{non_exist} we shall first show that the non-degeneracy
condition \eq{non_deg}, together with the TVD property, implies {\em strict} non-degeneracy
(cf.\ Theorem \ref{repr}):
\begin{definition}
	A scalar field $\vp:\mathcal H\to \RR$ is called {\em strictly non-degenerate} provided
	the set
	\beq\label{non_deg_strict} 
		\{\,(r,s)\in\mathcal H\,\,|\,\, \del_r\vp(r,s)\neq 0\text{ and } 
		\del_s\vp(r,s)\neq 0\,\} \quad\text{is dense in $\mathcal H$.}
	\eeq
\end{definition}
\begin{lemma}\label{non_strictly} 
	Any $C^2$-smooth, non-degenerate Glimm-type TVD field
	for \eq{tau}-\eq{p_syst} is necessarily strictly non-degenerate.
\end{lemma}
\begin{proof}
	Assume for contradiction that $\vp$ is non-degenerate but not strictly non-degenerate.
	By definition there then exists an open set $\mathcal A\subset\mathcal H$
	such that for any point $(r,s)\in \mathcal A$, either $ \del_r\vp(r,s)=0$ or $\del_s\vp(r,s)= 0.$ 
	As $\vp$ is non-degenerate, $\mathcal A$ must meet the open set $\mathcal H_\vp$.
	Thus, there exits $(r^*,s^*)\in \mathcal A\cap\mathcal H_\vp$, with 
	$\del_s\vp(r^*,s^*)\neq 0$, say. By smoothness of $\vp$ it follows that there is an open rectangle  
	$R=(r_1,r_2)\times (s_1,s_2)\subset\mathcal A\cap \mathcal H_\vp$ where $\del_s\vp(r,s)\neq 0$, 
	and at the same time $\del_r\vp(r,s)= 0$ (since $R\subset\mathcal A$).
	Thus, $\vp(r,s)=\theta(s)$ for $(r,s)\in (r_1,r_2)\times (s_1,s_2)$, and $\theta'(s)\neq 0$ 
	for $s\in(s_1,s_2)$.
	
	Now choose $(\bar r,\bar s)\in R$ and consider $\ba S\ba S$-interactions with left 
	state $(\bar r,\bar s)$. Let the incoming waves be so weak that all $(r,s)$-states in the solution
	belongs to $R$, and also such that $\sgn \theta'(s)=\sgn\theta'(\bar s)$ for all 
	$s$-values in the solution; this is possible by smoothness of $\vp$. The outgoing 
	wave-pattern is $\ba S\fa R$. By considering the wave-curves in the $(r,s)$-plane 
	it follows that $x\mapsto\theta(s(t,x))$ is strictly monotone at times $t$ before 
	the interaction, whereas it is non-monotone after interaction. 
	We conclude that $\vp$ cannot be a Glimm-type TVD field.
	If instead $\del_r\vp(r^*,s^*)\neq 0$, then a similar argument with weak $\fa S\fa S$-interactions
	yields the same conclusion.
\end{proof}	
As a consequence of Lemma \ref{non_strictly}, to establish Theorem \ref{non_exist}, we only 
need to prove the non-existence of strictly non-degenerate, Glimm-type TVD fields. To do so
we assume for contradiction that $\vp$ is a $C^2$-smooth, strictly non-degenerate 
Glimm-type TVD field. From Theorem \ref{repr} we 
know that $\vp$ admits one representation of the form
\beq\label{repr2}
	\vp(r,s)=\theta(s)-\psi(r)\qquad \text{on all of $\mathcal H$.}
\eeq
To reach a contradiction we proceed in two steps:
\begin{itemize}
	\item[] {\bf Step 1:} By considering the change in $\var\vp$ across $\ba S\ba S$- and 
	$\fa S\fa S$-interactions, we obtain that $\theta$ and $\psi$ must have the same, strict monotonicity.
	(See Proposition \ref{prop_1}.)
	\item[]
	\item[] {\bf Step 2:} Then, assuming without loss of generality that $\theta$ and $\psi$ are both strictly 
	increasing, there are three possibilities: $\exists a\in\RR$ where $\theta'(a)>\psi'(a)$, vice versa, 
	or $\theta$ and $\psi$ differ by a constant. For each case we analyze carefully chosen 
	interactions and obtain that $\var\vp$ increases across these, contradicting the assumed TVD property.
\end{itemize}
\subsection{Step 1: Monotonicity of $\theta$ and $\psi$}
In what follows we use the following notation:
\begin{itemize}
	\item $\Delta \var  \vp$ denotes the change in spatial variation of $\vp$ 
	across a pairwise wave interaction; see \eq{var_phi_before}, \eq{var_phi_after}, and \eq{dlt_tv}.
	\item  For any function $h\equiv h(s,r)$ and any $\fa S$, $\ba S$, $\fa R$, or $\ba R$ wave with 
	associated $\xi$-ratio $q$ (see \eq{bkwd_wave}-\eq{frwd_wave}), we set 
	$\Delta_q{h}:=h_{right}-h_{left}$, where $h_{right}$ and $h_{left}$ denote the values of $h$ on the right and on the left  
	of the wave, respectively. 
\end{itemize}
\begin{lemma}\label{change_s_r} We have:
\begin{itemize}
	\item[(i)] $\Delta_q{s}<0$ and $\Delta_q{r}<0$ for $\fa S$-waves and $\ba S$-waves;
	\item[(ii)] $\Delta_q{s}>0$ and $\Delta_q{r}=0$ for $\fa R$-waves;
	\item[(iii)] $\Delta_q{r}>0$ and $\Delta_q{s}=0$ for $\ba R$-waves;
\end{itemize}
\end{lemma}
\begin{proof}
	Immediate from the form of the wave curves in the Riemann coordinate plane (Figure 3).
\end{proof}
The following result narrows down the possible structure of strictly non-degenerate TVD fields in terms 
of the monotonicity of its two parts. 
\begin{proposition}\label{prop_1}
	Let $\vp$ be a strictly non-degenerate $C^2$-smooth TVD field of Glimm type 
	for \eq{tau}-\eq{p_syst}, and thus admitting a representation of the form 
	$\vp(r,s)=\theta(s)-\psi(r)$ on all of $\mathcal H$. 
	Then $\theta(s)$ and $\psi(r)$ are either both strictly increasing functions or both are 
	strictly decreasing functions. 
\end{proposition}
\noindent The proof of Proposition \ref{prop_1} will follow from Lemma \ref{lemma1/2} - Lemma \ref{new1}.
\begin{lemma}\label{lemma1/2}
	With the notation and assumptions of Proposition \ref{prop_1}  we have that
	neither $\theta'$ nor $\psi'$ can vanish identically on any open interval.
\end{lemma}
\begin{proof}
	Suppose for contradiction that $\psi'\equiv 0$ on an $r$-interval $(r_1,r_2)$. 
	By strict non-degeneracy \eq{non_deg}, there exist an $\bar s\in (r_1,r_2)$ 
	with $\theta'(\bar s)\neq 0$. We can now argue as in the proof of 
	Lemma \ref{non_strictly} to see that the variation of $\vp$ along the solution of weak
	$\ba S\ba S$-interactions with left state $(\bar r,\bar s)$, must necessarily increase.
	Thus, on any open $r$-interval, $\psi'\not\equiv 0$. A similar argument shows that 
	$\theta'\not\equiv 0$ on any open $s$-interval.
\end{proof}

\begin{lemma}\label{lemma1}
	With the notation and assumptions of Proposition \ref{prop_1} we have that
	\beq\label{poss_1}
		\psi'(r)\theta'(s)\geq 0\quad\quad\text{whenever $(r,s)\in\mathcal H$.}
	\eeq
\end{lemma}
\begin{proof}
	If not there exist a point $(\bar r,\bar s)\in\mathcal H$ such that either
	\beq\label{imposs_1}
		\psi'(\bar r)< 0 \quad\text{and}\quad\theta'(\bar s)> 0,
	\eeq
	or 
	\beq\label{imposs_2}
		\psi'(\bar r)>0\quad\text{and}\quad\theta'(\bar s)< 0.
	\eeq
	Suppose \eq{imposs_1} holds. Then, by continuity, there are intervals 
	$(r_1,\, r_2)\ni \bar r$ and $(s_1,\, s_2)\ni\bar s$ with
	\[\psi'(r)<0 \text{ on $(r_1,\, r_2)$}\quad 
	\text{and} \quad \theta'(s)>0\text{ on $(s_1,\, s_2)$.}\]
	If necessary we reduce the lengths of these intervals so as to have 
	$(r_1,\, r_2)\times (s_1,\, s_2)\subset \mathcal H$.
	We then let $(\bar r,\bar s)$ be the left state in an $\ba S \ba S$-interaction and choose 
	the strengths of the two incoming shocks sufficiently weak to guarantee that all 
	$(r,s)$-states in the solution of the interaction belong to $(r_1,\, r_2)\times (s_1,\, s_2)$.
	Let $x$, $y$, $B$ and $F$ denote the $\xi$-ratios (see \eq{xi_ratio}) across the incoming left shock, 
	incoming right shock, outgoing backward shock, and outgoing forward rarefaction, respectively. See
	the right diagram in Figure 4. By parts (i) and (ii) of Lemma \ref{change_s_r} we have
	\begin{align}
		\Delta \var  \vp&=|\Delta_B \vp|+|\Delta_F \vp|-|\Delta_x \vp|-|\Delta_y \vp |\nonumber\\
		&= |\Delta_B\theta(s)-\Delta_B \psi(r)|+|\Delta_F \theta(s)-\cancelto{0}{\Delta_F\psi(r)}|
		-|\Delta_x \theta(s)-\Delta_x\psi(r)|-|\Delta_y \theta(s)-\Delta_y\psi(r) |\nonumber\\
		&=-\Delta_B \theta+\Delta_B \psi+\Delta_F \theta+\Delta_x \theta-\Delta_x \psi
		+\Delta_y \theta -\Delta_y \psi \nonumber\\
		&=2\Delta_F \theta >0,
	\end{align}
	where we have used that 
	\beq\label{SS_relns}
		\Delta_B \theta(s)+\Delta_F \theta(s)=\Delta_x \theta(s)+\Delta_y \theta(s)\qquad\text{and}\qquad
		\Delta_B\psi(r) =\Delta_x \psi(r)+\Delta_y \psi(r).
	\eeq
	This contradicts the assumed TVD property of $\vp$. 
	A similar argument (again for an $\ba S \ba S$-interaction) shows that \eq{imposs_2} contradicts the 
	assumed TVD property.
\end{proof}	
\begin{lemma}\label{lemma2}
	With the notation and assumptions of Proposition \ref{prop_1} we have that
	$\theta(s)$ and $\psi(r)$ cannot change their monotonicity:
	\beq\label{lemma21}
		\forall\, r_1,\, r_2\in\RR:\qquad \psi'(r_1)\psi'(r_2)\geq 0,
	\eeq
	and 
	\beq\label{lemma22}
		\forall\, s_1,\, s_2\in\RR:\qquad \theta'(s_1)\theta'(s_2)\geq 0.
	\eeq
\end{lemma}
\begin{proof}
	For contradiction assume that \eq{lemma21} does not hold: there are $r_1,\, r_2\in\RR$ 
	with $r_1<r_2$ and such that $\psi'(r_1)\psi'(r_2)< 0$. Then, either
	\beq\label{opt1}
		\psi'(r_1)<0<\psi'(r_2),
	\eeq
	or
	\beq\label{opt2}
		\psi'(r_2)<0<\psi'(r_1).
	\eeq
	Since $\theta'$ cannot vanish identically on $(r_2,\infty)$ (by Lemma \ref{lemma1/2}), 
	there exists $\hat s\in (r_2,\infty)$ with $\theta'(\hat s)\neq 0$. We then have:
	\begin{itemize}
		\item if $\theta'(\hat s)>0$ and \eq{opt1} holds, or if $\theta'(\hat s)<0$ and \eq{opt2} holds,
		then the point $(r_1,\hat s)\in\mathcal H$ yields a contradiction with Lemma \ref{lemma1};
		\item if $\theta'(\hat s)<0$ and \eq{opt1} holds, or if $\theta'(\hat s)>0$ and \eq{opt2} holds,
		then the point $(r_2,\hat s)\in\mathcal H$ yields a contradiction with Lemma \ref{lemma1}.
	\end{itemize}
	This shows that \eq{lemma21} must hold. An analogous argument shows that \eq{lemma22}
	holds.
\end{proof}	
\begin{lemma}\label{new1}
	With the notation and assumptions of Proposition \ref{prop_1} we have that $\psi(r)$
	and $\theta(s)$ are both strictly monotonic.
\end{lemma}
\begin{proof}	
	By Lemma \ref{lemma2} $\psi$ is a monotone function.
	Lemma \ref{lemma1/2} then gives that $\psi$ cannot take the same value twice and it 
	is thus strictly monotonic. The same argument applies to $\theta$.
\end{proof}	
\noindent{\bf Proof of Proposition \ref{prop_1}:} 
Without loss of generality assume that $\psi(r)$ is strictly increasing. 
If $\theta(s)$ is not strictly increasing then it is strictly decreasing, by Lemma \ref{new1}.
Thus there exists an open $s$-interval $I$ such that $\theta'(s)<0$ for $s\in I$. Since $\psi(r)$ is 
strictly increasing there is an $\bar r\in I$ with $\psi'(\bar r)>0$.
But then $\psi'(\bar r)\theta'(s)<0$ for $s\in I,\, s>\bar r$, contradicting Lemma \ref {lemma1}.
A similar argument applies if $\psi(r)$ is strictly decreasing. 
This concludes the proof of Proposition \ref{prop_1}.
\qed

\medskip

\subsection{Step 2: Completing the proof of Theorem \ref{non_exist}}\label{2nd_part}
As already indicated we shall assume, for contradiction, that $\vp:\mathcal H\to \RR$ is 
a strictly non-degenerate $C^2$-smooth TVD field of Glimm type. 
Theorem \ref{repr} and Proposition \ref{prop_1} provide functions $\psi(r)$ and $\theta(s)$
with the same strict monotonicity, and such that the representation \eq{repr2} holds throughout
$\mathcal H$.

Since $\var  \vp=\var (-\vp)$ we need only consider the case where both $\psi(r)$ and 
$\theta(s)$ are strictly increasing. There are then three mutually exclusive possibilities: 
\begin{itemize}
	\item[] {\bf Case 1:} \quad  $\exists\, a\in \RR$ such that $\theta'(a)>\psi'(a)>0$, or 
	\item[] {\bf Case 2:} \quad  $\exists\, a\in \RR$ such that $\psi'(a)>\theta'(a)>0$, or
	\item[] {\bf Case 3:} \quad  $\psi=\theta+C$ for some constant $C$.
\end{itemize}
In each case we shall show that there is an interaction for which $\Delta \var \vp>0$, 
contradicting the assumed TVD property of $\vp$. This will then complete the proof 
of Theorem \ref{non_exist}.
The interactions we consider are $\ba S\ba S$, $\fa S\fa S$, and $\ba R\ba S$ for 
each of Cases 1, 2, and 3, respectively. 

\medskip

	\paragraph{{\bf Case 1}}
By continuity there is an open interval $J$ and constants $M>\delta>0$ such that 
\beq\label{geng_case1_0}
	\theta^{\prime}(s)>M>\psi^\prime(r)+\delta>\delta\qquad\text{whenever $r,\, s\in J$.}
\eeq
We consider an $\ba S\ba S$-interaction in which we first choose the $\xi$-ratios 
$x$ and $y$ across the incoming left and right shocks, respectively (see Figure 4, right diagram). 
The $\xi$-ratios across the outgoing backward shock and the outgoing forward rarefaction are denoted $B$ 
and $F$, respectively. These are uniquely determined by $x$ and $y$ alone, according 
to \eq{overtaking1}-\eq{overtaking2}, and the results in Section \ref{interactions} show that
$x,\, y,\, F,\, B>1$.  The far left state $(\bar u,\, \bar \xi)$ will then be chosen appropriately; see Remark \ref{det}.  

By \eq{bkwd_wave} the 
$(u,\, \xi)$-state between the two incoming waves is given by
\beq
	\tilde{\xi}=x\bar\xi\qquad \text{and}\qquad \tilde{u}=\bar u-\ba\phi(x)\bar \xi,
\eeq
and the $(u,\, \xi)$-state on the far right is given by
\beq
	\xi=y\tilde\xi\qquad \text{and}\qquad u=\tilde u-\ba\phi(y)\tilde \xi.
\eeq
We proceed to calculate the changes in $\vp(r,s)$ across the incoming waves $x$ and $y$.
According to \eq{delta_h} and \eq{delta_k} we have 
\begin{align}
	\Delta_x\vp &= \Delta_x \theta-\Delta_x \psi\nonumber\\
	&=-\bar \xi \int_1^x \Big[ \theta^\prime \big( \bar u-\ba\phi(\sigma)\bar \xi+ \kappa \sigma\bar\xi \big)
	(\ba\phi^\prime(\sigma)-\kappa  ) -  \psi^{\prime}\big( \bar u-\ba\phi(\sigma)\bar \xi- \kappa \sigma\bar\xi) 
	\big(\ba\phi^\prime(\sigma)+\kappa  )\Big] \, d\sigma,\label{geng_case1_1}
\end{align}
and 
\begin{align}
	\Delta_y\vp &= \Delta_y \theta-\Delta_y \psi\nonumber\\
	&=-\tilde \xi \int_1^y \Big[ \theta^\prime \big( \tilde u-\ba\phi(\sigma)\tilde \xi+ \kappa \sigma\tilde\xi \big)
	(\ba\phi^\prime(\sigma)-\kappa  ) -  \psi^{\prime}\big( \tilde u-\ba\phi(\sigma)\tilde \xi- \kappa \sigma\tilde\xi) 
	\big(\ba\phi^\prime(\sigma)+\kappa  )\Big]\, d\sigma.\label{geng_case1_2}
\end{align}
We fix $x,\, y>1$ such that
\beq\label{geng_case1_3}
	\ba\phi(x)>\kappa\left(\textstyle\frac{2M}{\delta}-1\right)(x-1) \qquad\text{and}\qquad
	\ba\phi(y)>\kappa\left(\textstyle\frac{2M}{\delta}-1\right)(y-1),
\eeq
which is possible according to \eq{phi_prime_infty}${}_1$ since $M>\delta$.
Next, fix any $\bar u\in J$ and $\bar\xi>0$ so small that
\beq\label{small_1}
	\bar u-\ba\phi(\sigma)\bar \xi\pm\kappa\sigma\bar \xi\in J\qquad\text{for all $\sigma\in[1,x]$,}
\eeq
and
\beq\label{small_2}
	\tilde u-\ba\phi(\sigma)\tilde \xi\pm\kappa\sigma\tilde \xi
	=\bar u-\left(\ba\phi(x)+\ba\phi(\sigma)x\mp\kappa\sigma x\right)\bar \xi \in J
	\qquad\text{for all $\sigma\in[1,y]$.}
\eeq
This guarantees that the arguments of $\theta'$ and $\psi'$ in \eq{geng_case1_1} 
and \eq{geng_case1_2} are evaluated at points in $J$.
Finally, the state $(\hat\xi,\hat u)$ between the two outgoing waves is given by \eq{bkwd_wave} as
\[\hat \xi=B\bar\xi \qquad\text{and}\qquad \hat u=\bar u-\ba\phi(B)\bar \xi,\]
where $B$ is determined entirely by $x$ and $y$ (see Remark \ref{det}). The corresponding $r$ 
and $s$ values are $\bar u-\ba\phi(B)\bar\xi\pm\kappa\bar\xi$.

We now combine this with \eq{geng_case1_0} to determine the signs of $\Delta_x\vp$ and 
$\Delta_y\vp$. Applying \eq{geng_case1_0} and \eq{small_1}-\eq{small_2} in \eq{geng_case1_1} 
and \eq{geng_case1_2}, and using that $\ba\phi'(\sigma)>\kappa$ for $\sigma>1$, we get 
\[\Delta_x\vp<\bar \xi \int_1^x \kappa(2M-\delta)-\delta\ba\phi'(\sigma)\,d\sigma 
= \bar\xi\left[\kappa(2M-\delta)(x-1)-\delta\ba\phi(x)\right],\]
and
\[\Delta_y\vp<\tilde \xi \int_1^y \kappa(2M-\delta)-\delta\ba\phi'(\sigma)\,d\sigma 
= \tilde\xi\left[\kappa(2M-\delta)(y-1)-\delta\ba\phi(y)\right].\]
By \eq{geng_case1_3} we therefore have
\[\Delta_x\vp,\, \Delta_y\vp<0.\]
Recalling Lemma \ref{change_s_r} and using that $\theta$ and $\psi$ are both strictly increasing, 
we conclude that
\begin{align*}
	\Delta \var \vp
	&=|\Delta_B\vp|+|\Delta_F\vp|-|\Delta_x\vp|-|\Delta_y\vp|\nonumber\\
	&= |\Delta_B\theta-\Delta_B \psi|+|\Delta_F \theta-\cancelto{0}{\Delta_F\psi}|
	+\Delta_x \vp+\Delta_y\vp\nonumber\\
	&=|\Delta_B\theta-\Delta_B\psi|+\Delta_F \theta+\Delta_x \theta-\Delta_x\psi
	+\Delta_y \theta-\Delta_y\psi  \nonumber\\
	&= | \Delta_B\theta-\Delta_B\psi | + (\Delta_B\theta-\Delta_B\psi) + 2\Delta_F\theta \\
	&\geq 2\Delta_F\theta > 0,
\end{align*}
where we have used \eq{SS_relns}. We have thus provided an $\ba S\ba S$-interaction in 
which the total variation of $\vp$ increases. That is, Case 1 cannot apply if $\vp$ is a strictly non-degenerate 
TVD field of Glimm type.

\bigskip

\paragraph{{\bf Case 2:}} The argument is similar to that of Case 1. For completeness we include the details. 
By continuity there is an open interval $J$ and constants $M>\delta>0$ such that 
\beq\label{geng_case2_0}
	\psi'(r)>M>\theta'(s)+\delta\qquad\text{whenever $r,\, s\in J$.}
\eeq
We then consider an $\fa S\fa S$-interaction for which we first choose the $\xi$-ratios 
$x$ and $y$ across the incoming left and right shocks, respectively. The $\xi$-ratios 
across the outgoing backward rarefaction and the outgoing forward shock are denoted $B$ 
and $F$, respectively. As before these are uniquely determined by $x$ and $y$.
By the results in Section \ref{interactions} we have $0<x,\, y,\, F,\, B<1$. 
The far left state $(\bar u,\, \bar \xi)$ will be chosen appropriately; see Remark \ref{det}.

By \eq{frwd_wave} the 
$(u,\, \xi)$-state between the two incoming waves is given by
\beq
	\tilde{\xi}=x\bar\xi\qquad \text{and}\qquad \tilde{u}=\bar u+\fa\phi(x)\bar \xi,
\eeq
and the $(u,\, \xi)$-state on the far right is given by
\beq
	\xi=y\tilde\xi\qquad \text{and}\qquad u=\tilde u+\fa\phi(y)\tilde \xi.
\eeq
We proceed to calculate the changes in $\vp(r,s)$ across the incoming waves $x$ and $y$.
By using expressions similar to those in \eq{delta_h} and \eq{delta_k}, but now for changes 
across forward waves, we have 
\begin{align}\label{geng_new_4}
	\Delta_x \vp &= \Delta_x \theta-\Delta_x \psi\nonumber\\
	&=\bar \xi \int_x^{1} \Big[-\theta^\prime \big( \bar u+\fa\phi(\sigma)\bar \xi+ 
	\kappa \sigma\bar\xi \big)(\fa\phi^\prime(\sigma)+\kappa ) + \psi^{\prime}\big( 
	\bar u+\fa\phi(\sigma)\bar \xi- \kappa \sigma\bar\xi) \big(\fa\phi^\prime(\sigma)-\kappa )\Big]\, d\sigma,
\end{align}
and
\begin{align}\label{geng_new_5}
	\Delta_y \vp &= \Delta_x \theta-\Delta_x \psi\nonumber\\
	&=\tilde \xi \int_y^{1} \Big[- \theta^\prime \big( \tilde u+\fa\phi(\sigma)\tilde \xi+ 
	\kappa \sigma\tilde\xi \big)(\fa\phi^\prime(\sigma)+\kappa ) + \psi^{\prime}\big( \tilde u
	+\fa\phi(\sigma)\tilde \xi- \kappa t\tilde\xi) \big(\fa\phi^\prime(\sigma)-\kappa )\Big]\, d\sigma.
\end{align}
Using \eq{phi_prime_infty}${}_2$ and $M>\delta$, we fix $x,\, y\in (0,1)$ such that
\beq\label{geng_case2_3}
	\fa\phi(x)<\kappa\left(1-\textstyle\frac{2M}{\delta}\right)(1-x) \qquad\text{and}\qquad
	\fa\phi(y)<\kappa\left(1-\textstyle\frac{2M}{\delta}\right)(1-y).
\eeq
Next, fix any $\bar u\in J$ and choose $\bar\xi>0$ so small that
\beq\label{small_3}
	\bar u+\fa\phi(\sigma)\bar \xi\pm\kappa\sigma\bar \xi\in J\qquad\text{for all $\sigma\in[x,1]$,}
\eeq
and
\beq\label{small_4}
	\tilde u+\fa\phi(\sigma)\tilde \xi\pm\kappa\sigma\tilde \xi
	=\bar u+\left(\fa\phi(x)+\fa\phi(\sigma)x\pm\kappa\sigma x\right)\bar \xi \in J
	\qquad\text{for all $\sigma\in[y,1]$.}
\eeq
This guarantees that the arguments of $\theta'$ and $\psi'$ in \eq{geng_new_4} 
and \eq{geng_new_5} are evaluated at points in $J$.
Finally, the state $(\hat\xi,\hat u)$ between the two outgoing waves is given by \eq{bkwd_wave} as
\[\hat \xi=B\bar\xi \qquad\text{and}\qquad \hat u=\bar u-\ba\phi(B)\bar \xi,\]
where $B$ is determined entirely by $x$ and $y$ (see Remark \ref{det}). The corresponding $r$- 
and $s$-values are $\bar u-\ba\phi(B)\bar\xi\pm\kappa\bar\xi$.

We now combine this with \eq{geng_case2_0} to determine the signs of $\Delta_x\vp$ and 
$\Delta_y\vp$. Applying \eq{geng_case2_0} and \eq{small_3}-\eq{small_4} in \eq{geng_new_4} 
and \eq{geng_new_5}, and using that $\fa\phi'(\sigma)>\kappa$ for $0<\sigma<1$, we get 
\[\Delta_x\vp>\bar \xi \int_x^1 \kappa(\delta-2M)+\delta\fa\phi'(\sigma)\,d\sigma 
= \bar\xi\left[\kappa(\delta-2M)(1-x)-\delta\fa\phi(x)\right],\]
and
\[\Delta_y\vp>\tilde \xi \int_y^1 \kappa(\delta-2M)+\delta\fa\phi'(\sigma)\,d\sigma 
= \tilde\xi\left[\kappa(\delta-2M)(1-y)-\delta\fa\phi(y)\right].\]
By \eq{geng_case2_3} we therefore have
\[\Delta_x\vp,\, \Delta_y\vp>0.\]
Recalling Lemma \ref{change_s_r} and using that $\theta$ and $\psi$ are both strictly increasing, 
we conclude that
\begin{align*}
	\Delta \var (\vp) &=|\Delta_B\vp|+|\Delta_F\vp|-|\Delta_x\vp|-|\Delta_y\vp|\nonumber\\
	&=|\cancelto{0}{\Delta_B\theta}-\Delta_B \psi|+| \Delta_F\theta-\Delta_F\psi| -
	\Delta_x \vp-\Delta_y\vp\nonumber\\
	&=\Delta_B \psi+|\Delta_F\theta-\Delta_F\psi |
	-\Delta_x \theta+\Delta_x\psi-\Delta_y\theta+\Delta_y\psi\\
	&=2\Delta_B\psi+| \Delta_F\theta-\Delta_F\psi | - (\Delta_F\theta-\Delta_F\psi) >0\\
	&\geq 2\Delta_B\psi>0,
\end{align*}
where we have used
\[\Delta_x \theta(s)+\Delta_y \theta(s)=\Delta_F \theta(s)\qquad\text{and}\qquad
\Delta_x \psi(r)+\Delta_y \psi(r) =\Delta_F \psi(r)+\Delta_B \psi(r).\]
This provides an interaction in which the total variation of $\vp$ increases, such that Case 2 cannot
apply if $\vp$ is a strictly non-degenerate TVD field of Glimm type.

\medskip

The only remaining possibility for a strictly non-degenerate TVD field of Glimm type $\vp(r,s)=\theta(s)-\psi(r)$, 
is that Case 3 holds: $\theta(s)$ and $\psi(r)$ differ by a constant. Note that the Liu functional $L(t)$, defined 
in \eq{liu_fncl} for $\gamma=1$, falls into this case with $\theta\equiv \psi=\text{id}$.
However, as we show next, for $\gamma>1$ there are interactions in which the total variation of such 
$\vp$ must necessarily increase.

\bigskip

\paragraph{{\bf Case 3:}}
We consider particular $\ba R\ba S$ interactions. As in Case 1 and Case 2 we first fix the 
$\xi$-ratios across the incoming waves, and then suitably choose the far left state in the 
interaction. We will make use of the observation in Section \ref{overtaking_interactions} 
that the outcome of an $\ba R\ba S$ interaction is $\ba S\fa S$ provided the incoming 
$\ba S$-wave is sufficiently strong relative to the incoming $\ba R$-wave.

We start by fixing any finite, open $s$-interval $I\subset \RR$. Since $\theta$ is assumed to be 
$C^2$ smooth and strictly increasing, there are positive constants $M_U$, $N_L$, 
and $N_U$ such that 
\beq\label{geng_case3_0}
	|\theta''(s)|<M_U,\qquad N_U>\theta'(s)>N_L>0,\qquad \forall\, s\in I.
\eeq
As before we denote the far left state in the interaction by $(\bar u,\, \bar \xi)$, while
$x<1$ and $y>1$ denote the $\xi$-ratios across the incoming rarefaction and shock 
waves, respectively. Again, $B$ and $F$ denote the $\xi$-ratios 
across the outgoing backward and forward waves. We recall that these are functions
of $x$ and $y$ alone. In particular, $B=B(x,y)$ is given as the solution of \eq{overtaking_B},
and we proceed to Taylor expand it about $x=1$.
Differentiating \eq{overtaking_B} with respect to $x$, evaluating at $x=1$, solving for $\del_x B(1,y)$, 
and using that $B(1,y)=y$, we obtain the expansion 
\begin{align}
	B(x,y)&=y+\left(\frac{\ba\phi(y)+\kappa(y+1)}{\ba\phi'(y)+\kappa}\right)(x-1)+O_y\!\left((x-1)^2\right)\nonumber\\
	&=xy+\left(y-\frac{\ba\phi(y)+\kappa(y+1)}{\ba\phi'(y)+\kappa}\right)(1-x)+O_y\!\left((x-1)^2\right),\label{B_tayl}
\end{align}
where $O_y$ indicates that the last term depends also on $y$. We now fix any $\veps>0$ and then 
$y>1$ so large that 
\beq\label{phi_ineq}
	y-\left(\frac{\ba\phi(y)+\kappa(y+1)}{\ba\phi'(y)+\kappa}\right)>\frac{N_U+2\veps}{N_L},
\eeq
where we use the fact recorded earlier in \eq{phi_limit}. We then fix $x<1$ sufficiently close 
to $1$ to guarantee that $B=B(x,y)>1$, such that the outcome of the $\ba R\ba S$-interaction is $\ba S\fa S$. 
If necessary we further increase $x$ towards $1$ so as to guarantee that $\veps(1-x)/N_L+O_y\!\left((x-1)^2\right)>0$.
With these choices we apply \eq{phi_ineq} in \eq{B_tayl} to obtain
\beq\label{intermed_ineq}
	B>xy+\frac{N_U+2\veps}{N_L}(1-x)+O_y\!\left((x-1)^2\right)>xy+\frac{N_U+\veps}{N_L}(1-x).
\eeq
At this point the strengths $x<1$, $y>1$, $B=B(x,y)>1$, and $F(x,y)<1$, are all fixed. 
We finally rewrite \eq{intermed_ineq}, using \eq{overtaking1}, as
\beq\label{B_ineq_2}
	N_LB(1-F)>(N_U+\veps)(1-x).
\eeq

\medskip

We proceed to compute the changes in $\vp(r,s)=\theta(s)-\theta(r)+Const.$ across each of the 
waves in the interaction. 
Applying the identities \eq{delta_h} and \eq{delta_k} for changes across backward waves
we obtain the following expressions
\beq
	\Delta_x \vp = \cancelto{0}{\Delta_x\theta(s)} - \Delta_x \theta(r)
	= -\bar\xi\int_x^1\big(\ba\phi'(\sigma)+\kappa\big)
	\theta'(\bar u -\ba \phi(\sigma)\bar \xi - \kappa \sigma\bar\xi)\, d\sigma,\label{expr_1}
\eeq
\begin{align}
	\Delta_y \vp &= \Delta_y\theta(s) - \Delta_y \theta(r)\nn\\
	&= -\tilde\xi\int_1^y\big(\ba\phi'(\sigma)-\kappa\big) \theta'(\tilde u -\ba \phi(\sigma)\tilde \xi + \kappa \sigma\tilde\xi)
	- \theta'(\tilde u -\ba \phi(\sigma)\tilde \xi - \kappa \sigma\tilde\xi)
	\big(\ba\phi'(\sigma)+\kappa\big)\, d\sigma\nn\\
	&= -\tilde\xi\left[\int_1^y \big(\ba\phi'(\sigma)-\kappa\big) 
	\Big[\int_{\tilde u -\ba \phi(\sigma)\tilde \xi - \kappa \sigma\tilde\xi}^{\tilde u -\ba \phi(\sigma)\tilde\xi+\kappa \sigma\tilde\xi}
	\nqquad\theta''(\mu)\,d\mu\Big]\, d\sigma
	-2\kappa \int_1^y \theta'(\tilde u -\ba \phi(\sigma)\tilde \xi - \kappa \sigma\tilde\xi)\, d\sigma\right],\label{expr_2}
\end{align}
where $\tilde{u}=\bar u-\ba\phi(x)\bar \xi$, $\tilde{\xi}=x\bar\xi$ are the $u$- and $\xi$-values of the state between the 
two incoming waves. Similarly,
\begin{align}
	\Delta_B \vp &= \Delta_B\theta(s) - \Delta_B \theta(r)\nn\\
	&=-\bar\xi\left[\int_1^B \big(\ba\phi'(\sigma)-\kappa\big) 
	\Big[\int_{\bar u -\ba \phi(\sigma)\bar \xi - \kappa \sigma\bar\xi}^{\bar u -\ba \phi(\sigma)\bar\xi+\kappa \sigma\bar\xi}
	\nqquad\theta''(\mu)\,d\mu\Big]\, d\sigma
	-2\kappa \int_1^B \theta'(\bar u -\ba \phi(\sigma)\bar \xi - \kappa \sigma\bar\xi)\, d\sigma\right],\label{expr_3}
\end{align}
The corresponding identities for forward waves yield
\begin{align}
	\Delta_F \vp &= \Delta_F\theta(s) - \Delta_F \theta(r)\nn\\
	&=-\hat\xi\left[\int_F^1 \big(\fa\phi'(\sigma)+\kappa\big) 
	\Big[\int_{\hat u +\fa \phi(\sigma)\hat \xi - \kappa \sigma\hat\xi}^{\hat u +\fa \phi(\sigma)\hat\xi+\kappa \sigma\hat\xi}
	\nqquad\theta''(\mu)\,d\mu\Big]\, d\sigma
	+2\kappa \int_F^1 \theta'(\hat u +\fa \phi(\sigma)\hat \xi - \kappa \sigma\hat\xi)\, d\sigma\right],\label{expr_4}
\end{align}
where
\[\hat{u}=\bar u-\ba\phi(B)\bar \xi,\qquad \hat{\xi}=B\bar\xi\]
are the $u$- and $\xi$-values of the state between the two outgoing waves.

It follows from the expressions for $\tilde \xi$, $\tilde u$, $\hat \xi$, $\hat u$, 
that for $\bar u\in I$ and $\bar\xi$ sufficiently small, all arguments of $\theta'$ and $\theta''$ 
occurring in \eq{expr_1}-\eq{expr_4}, and indeed all $r$- and $s$-values of all states involved 
in the interaction, belong to $I$. Next, it is clear from \eq{expr_1} that 
\beq\label{reln_1}
	\Delta_x\vp<0\,.
\eeq
Furthermore, it follows from \eq{geng_case3_0}, \eq{expr_2}-\eq{expr_3}, that 
\beq\label{reln_2}
	\Delta_y \vp>0\qquad\text{and}\qquad  \Delta_B\vp >0,
\eeq
for $\bar\xi$ sufficiently small. Finally, by using \eq{geng_case3_0}, \eq{expr_1}, and \eq{expr_4},
we have that 
\beq\label{reln_3}
	\Delta_F\vp<\Delta_x\vp,
\eeq
provided $O(1)\bar\xi+N_LB(1-F)>N_U(1-x)$, where the $O(1)$ term depends only on $B$, $F$, $M_U$, and $\kappa$. 
According to \eq{B_ineq_2} this latter condition is satisfied for $\bar\xi$ sufficiently small.
Hence, by \eq{reln_1}-\eq{reln_3} and Lemma \ref{change_s_r}, we obtain that whenever $\bar\xi$ is sufficiently 
small, then
\begin{align}
	\Delta \var  \vp&=|\Delta_B \vp|+|\Delta_F \vp|-|\Delta_x \vp|-|\Delta_y \vp |\nn\\
	&=\big\{\Delta_B \vp - \Delta_F \vp\big\}+\big\{\Delta_x \vp - \Delta_y \vp \big\}\nn\\
	&=\Big\{\big[\Delta_B \theta(s)-\Delta_B \theta(r)\big]-\big[\Delta_F \theta(s)
	-\Delta_F \theta(r)\big]\Big\}
	+\Big\{-\Delta_x \theta(r) -\big[\Delta_y \theta(s)-\Delta_y \theta(r)\big]\Big\} \nn\\
	&=-2\Delta_F \theta(s)-2\Delta_x \theta(r)+2\Delta_F \theta(r) \nn\\
	&= -2\Delta_F\vp + 2\Delta_x\vp>0,\label{pos}
\end{align}
where we have used the identities
\[\Delta_B \theta(s)=\Delta_y \theta(s)-\Delta_F \theta(s),\qquad 
\Delta_y \theta(r)=-\Delta_x \theta(r)+\Delta_B \theta(r)+\Delta_F \theta(r).\]

\medskip

\noindent{\bf Proof of Theorem \ref{non_exist}:} 
Taken together, Case 1, Case 2, and Case 3 show that the system \eq{tau}-\eq{p_syst}
for isentropic flow of an ideal, polytropic gas with $\gamma>1$, does not admit any strictly 
non-degenerate Glimm-type TVD field $\vp$.\qed

\section{Final remarks}
The non-existence of TVD fields for the $p$-system is similar in spirit to Temple's result
\cite{temple_85} that no metric on state space yields $L^1$-contraction for 
systems.\footnote{Temple's result applies to general $2\times 2$-systems 
and was proved by analyzing solutions without interactions.} 
We note that our analysis does not rule out decay for other types of Nishida-like 
functionals such as 
\[\mathfrak N(t):=\underset{\ba S}{\var}\, \vp_1(r(t)) + \underset{\fa S}{\var}\, \vp_2(s(t)),\]
for suitably chosen scalar fields $\vp_1$, $\vp_2$, say. Furthermore, our results do not rule out the existence of
``local'' TVD fields, i.e.\ scalar fields $\vp$ whose variation decays along all solutions with values
in a fixed compact subset of the no-vacuum set $\mathcal H$. However, it remains an open problem 
to prove or disprove that solutions of \eq{tau}-\eq{p_syst} are bounded away from vacuum unless a vacuum
appears immediately at time $0^+$. In connection to this we note that our proofs above utilize sufficiently strong 
interactions, sufficiently close to vacuum.


\end{document}